\documentclass[11pt]{amsart}

\usepackage{amssymb,amsmath,amscd}
\usepackage[utf8]{inputenc}
\usepackage[mathscr]{euscript}
\usepackage{bbm}
\usepackage[left=1.3in,right=1.3in,top=1in,bottom=1in]{geometry}

\date{}

\newcommand{\Hi}{{\mathcal H}}
\newcommand{\Z}{{\mathbbm Z}}
\newcommand{\R}{{\mathbbm R}}
\newcommand{\C}{{\mathbbm C}}

\newcommand{\T}{{\mathbbm T}}
\newcommand{\I}{{\mathbbm I}}

\newcommand{\ZL}{{\mathcal Z}}

\newcommand{\J}{{\mathscr{J}}}

\newcommand{\ac}{{\mathrm{ac}}}

\newcommand{\dd}{{\mathrm{d}}}

\newcommand{\Hd}{{\mathrm{H}}}

\newcommand{\CMV}{{\mathcal C}}

\newcommand{\EC}{{\mathrm{EC}}}

\DeclareMathOperator*{\slim}{s-lim}

\allowdisplaybreaks

\numberwithin{equation}{section}

\newcommand{\tr}{\mathrm{tr} }

\newcommand{\Leb}{{\mathrm{Leb}}}

\newcommand{\set}[1]{\left\{#1\right\}}
\newcommand{\eqdef}{\overset{\mathrm{def}}=}
\newcommand{\norm}[1]{\left\|#1\right\|}

\newtheorem{theorem}{Theorem}[section]
\newtheorem{lemma}[theorem]{Lemma}
\newtheorem{prop}[theorem]{Proposition}
\newtheorem{coro}[theorem]{Corollary}

\theoremstyle{definition}
\newtheorem{remark}[theorem]{Remark}
\theoremstyle{definition}
\newtheorem*{defi}{Definition}
\theoremstyle{definition}

\theoremstyle{definition}

\begin{document}
\title[Ballistic Transport for Limit-Periodic Jacobi Matrices]{Ballistic Transport for Limit-Periodic Jacobi Matrices with Applications to Quantum Many-Body Problems}

\author[J.\ Fillman]{Jake Fillman}

\address{Virginia Tech, 225 Stanger Street, Blacksburg, VA 24061}
 
\email{fillman@vt.edu}
\date{}

\maketitle

\begin{abstract}
We study Jacobi matrices that are uniformly approximated by periodic operators. We show that if the rate of approximation is sufficiently rapid, then the associated quantum dynamics are ballistic in a rather strong sense; namely, the (normalized) Heisenberg evolution of the position operator converges strongly to a self-adjoint operator that is injective on the space of absolutely summable sequences. In particular, this means that all transport exponents corresponding to well-localized initial states are equal to one. Our result may be applied to a class of quantum many-body problems. Specifically, we establish a lower bound on the Lieb--Robinson velocity for an isotropic XY spin chain on the integers with limit-periodic couplings.
\end{abstract}

\section{Introduction}

\subsection{Limit-Periodic Jacobi Matrices}

We will study the quantum dynamical characteristics of limit-periodic Jacobi matrices in the regime of absolutely continuous spectrum. By a \emph{Jacobi matrix}, we mean a bounded self-adjoint operator $\J = \J_{a,b}: \ell^2(\Z) \to \ell^2(\Z)$ defined by
\begin{equation} \label{eq:jacobidef}
(\J \varphi)_n
=
a_{n-1} \varphi_{n-1}  + b_n \varphi_n + a_n \varphi_{n+1},
\quad
n \in \Z, \; \varphi \in \ell^2(\Z),
\end{equation}
where $a,b \in \ell^\infty(\Z,\R)$ and $\inf_n a_n > 0$. In the present paper, our goal is to understand the associated quantum dynamics; that is to say, we want to probe asymptotic characteristics of solutions to the corresponding \emph{time-dependent Schr\"odinger equation}:
\begin{equation} \label{eq:tdschro}
i\frac{\partial \psi}{\partial t}
=
\J \psi,
\quad
\psi(0) = \varphi_0 \in \ell^2(\Z).
\end{equation}
Using the spectral theorem, one may explicitly solve \eqref{eq:tdschro} via
\[
\psi(t)
=
e^{-it\J}\varphi_0,
\quad
t \in \R.
\]
If there exists $q \in \Z_+$ for which $a_{n+q} = a_n$ and $b_{n+q}= b_n$ for all $n \in \Z$, we say that $\J$ is $q$-\emph{periodic}, and it is known that the dynamics described by the time-dependent Schr\"odinger equation~\eqref{eq:tdschro} are ballistic in a very strong sense \cite{DLY}; see also \cite{AK98} for the corresponding statement for continuum Schr\"odinger operators in $L^2(\R)$. Concretely, there exists a bounded self-adjoint operator $Q = Q_\J$  such that $\ker(Q_\J) = \set{0}$ and
\begin{equation} \label{eq:asch-knauf}
\slim_{t \to \infty} \frac{1}{t} X_\J(t)
=
Q_\J,
\end{equation}
where $X_\J(t)$ denotes the Heisenberg evolution of the position operator relative to $\J$, i.e.,
\[
X_\J(t)
\eqdef
e^{it\J} X e^{-it\J},
\quad
t \in \R.
\]
The position operator is an (unbounded) self-adjoint operator defined by
\begin{align*}
(X\varphi)_n 
& \eqdef
n \varphi_n,
\quad
n \in \Z, \\
& \varphi \in 
D(X)
\eqdef
\set{\varphi \in \ell^2(\Z) : \|X\varphi\|^2 = \sum_{n\in \Z} |n \varphi_n|^2 < \infty}.
\end{align*}

A particular consequence of \eqref{eq:asch-knauf} is that all of the moments of the position operator grow ballistically in time whenever $\J$ is periodic. More specifically, for $p > 0$ and $\varphi \in \ell^2(\Z)$, define
\[
|X|^p_\varphi(t)
\eqdef
\sum_{n \in \Z} (|n|^p + 1) \left| \langle \delta_n, e^{-it\J} \varphi \rangle \right|^2.
\] 
When considering such objects, we want to avoid trivialities that arise due to the divergence of moments, so we will restrict attention to states that are well-localized in the sense that they belong to the discrete Schwartz space
\[
\mathcal S
=
\mathcal S(\Z)
\eqdef
\set{\varphi \in \ell^2(\Z) : \sum_{n\in \Z}|n|^p | \varphi_n | < \infty \text{ for all } p > 0}.
\]
Thus, given $\varphi \in \mathcal S$, we want to measure the growth of $|X|^p_\varphi$ in time on a polynomial scale, so it makes sense to define
\[
\beta_\varphi^+(p)
\eqdef
\limsup_{t \to \infty} 
\frac{\log |X|_\varphi^p(t)}{p \log t},
\quad
\beta_\varphi^-(p)
\eqdef
\liminf_{t \to \infty} 
\frac{\log |X|_\varphi^p(t)}{p \log t}.
\]
The equation \eqref{eq:asch-knauf} (together with injectivity of $Q_\J$) immediately implies that $\beta_\varphi^+(p)= \beta_\varphi^-(p) = 1$ for all $p > 0$ and all $\varphi \in \mathcal S$.

There is substantial interest in determining estimates and exact values for these transport exponents in aperiodic lattice models. In general, one may probe solutions of \eqref{eq:tdschro} by analyzing the \emph{spectral measures} of $\J$, defined by
\[
\int_\R f(E) \, \dd\mu_{\varphi}(E)
=
\langle \varphi, f(\J) \varphi \rangle,
\quad
\varphi \in \ell^2(\Z),
\; f \in C_c(\R).
\] 
If $\varphi \in \mathcal S$ is a state such that the associated spectral measure $\mu_{\varphi}$ is absolutely continuous with respect to Lebesgue measure, then one immediately has what has been termed \emph{quasi--ballistic motion} in the literature. Specifically, one has $\beta^+_\varphi(p) = 1$ for all $p > 0$ \cite{Last96}. However, this approach to quantum dynamics has limitations; for example, it is not clear whether absolute continuity can preclude $\beta^-_\varphi(p) < 1$ for some $p > 0$, i.e., sub-ballistic transport on a subsequence of time scales. Notice that it is necessary to restrict the dimension to discuss the relationship between a.c.\ spectrum and ballistic motion, as there exist multidimensional models with a.c.\ spectrum and sub-ballistic (in fact, sub-diffusive) transport \cite{BelSchuba2000}. For a related characterization of the relationship between absolutely continuous spectrum and transport in terms of Thouless transport and Landauer--B\"uttiker conductance, see \cite{BJP,BJLP1,BJLP2,BJLP3,BJLP4}.

As another step towards answering the general question, it is therefore interesting to pursue special cases in which one knows that the Jacobi matrix enjoys purely absolutely continuous spectrum, and in which one has access to more specialized tools than are present in the general setting. Within the last year, several preprints have appeared that prove ballistic transport for several classes of \emph{quasi-periodic} operators with purely absolutely continuous spectrum \cite{kach,zhangzhao:ballistic,zhao:ballistic}. An intriguing special case of Jacobi matrices with purely a.c.\ spectrum is furnished by limit-periodic operators that are approximated by periodic Jacobi matrices exponentially quickly. 

\begin{defi}
A Jacobi matrix $\J$ is said to be \emph{limit-periodic} if there exists a sequence $\set{\J_n}_{n=1}^\infty$ consisting of periodic Jacobi matrices such that
\[
\lim_{n \to \infty} \norm{\J - \J_n}
=
0.
\]
Such operators (and their continuum counterparts) were studied heavily in the 1980's; see, for instance,~\cite{AS81,Chul81,egorova,MC84,Moser81,PT1,PT2,poschel83}. In recent years, a substantial amount of progress has been made in the study of spectral characteristics of such operators~\cite{avila,damanikgan1,damanikgan2,damanikgan3,DamGorLP,F14,FL15,gan2010,gankrueger}, as well as their unitary analogs, CMV matrices with limit-periodic Verblunsky coefficients~\cite{ong12}. However, outside the results of \cite{DLY}, relatively little is known regarding the quantum dynamics of limit-periodic Jacobi matrices.

Given $\eta > 0$, we say that a limit-periodic Jacobi matrix $\J$ is of \emph{exponential class} $\eta$ (denoted $\J \in \EC(\eta)$) if there exists a sequence of $q_n$-periodic Jacobi matrices $\J_n$ such that
\[
\lim_{n \to \infty} e^{\eta q_{n+1}} \|\J - \J_n\|
=
0,
\]
$q_n | q_{n+1}$, and $q_n \neq q_{n+1}$ for all $n$. We say that $\J \in \EC(\infty)$ if it is of exponential class $\eta$ for all $\eta > 0$~\cite{Chul81,egorova,PT1,PT2}. Recall that we have adopted the standing assumption $\inf_n a_n > 0$.
\end{defi}

If $\J \in \EC(\infty)$, it is known that $\J$ has purely absolutely continuous spectrum~\cite{egorova}, that its spectrum is homogeneous in the sense of Carleson~\cite{FL15}, and that $\J$ is reflectionless thereupon~\cite{egorova}. More generally, homogeneity of the spectrum and absolute continuity of the spectral measures of $\J \in \EC(\eta)$ with $\eta$ sufficiently large follow from the methods of \cite{FL15} and \cite{egorova}, respectively; see also~\cite[Chapter~4]{ChulAPONIS}. For the reader's convenience, we provide a detailed proof of absolutely continuous spectrum in Appendix~\ref{sec:acspec} by way of spectral homogeneity and some modern results in inverse spectral theory. In this paper, we prove that the dynamics associated to such a Jacobi matrix are ballistic in the sense of~\eqref{eq:asch-knauf}.

Throughout the paper, there will be various estimates that can be controlled as long as one has uniform bounds on the coefficients of $\J$. To that end, for every $R>0$, let us denote by $\mathcal J(R)$ the set of operators of the form $\J = \J_{a,b}$ defined by~\eqref{eq:jacobidef} such that $\|\J\| \leq R$ and $\inf_n a_n \geq R^{-1}$.

Our main result is an extension of~\cite[Theorem~1.6]{DLY} to operators in $\EC(\eta)$ with $\eta$ sufficiently large.

\begin{theorem} \label{pt:ballistic}
For every $R > 0$, there is a constant $\eta_0 = \eta_0(R)$ with the following property. If $\J \in \EC(\eta_0)$ and $\J \in \mathcal J(R)$, then the dynamics arising from \eqref{eq:tdschro} are ballistic in the sense that there exists a self-adjoint operator $Q_\J$ such that
\[
\slim_{t \to \infty} \frac{1}{t} X_\J(t)
=
Q_\J
\]
and $\ker(Q_\J) \cap \ell^1(\Z) = \set{0}$. In particular, $\beta^\pm_\varphi(p) = 1$ for all $\varphi \in \mathcal S$ and all $p > 0$.
\end{theorem}

\begin{coro}
If $\J \in \EC(\infty)$, then the associated quantum dynamics are ballistic in the sense of Theorem~\ref{pt:ballistic}.
\end{coro}

\begin{remark} 
A few remarks regarding Theorem~\ref{pt:ballistic} are in order:
\begin{enumerate}
\item To the best of our knowledge, Theorem~\ref{pt:ballistic} supplies the first explicit examples of aperiodic Jacobi matrices for which
\[
\lim_{t \to \infty} \frac{1}{t} \norm{X e^{-it\J} \varphi} \in (0,\infty)
\text{ for all }
\varphi \in \mathcal S.
\]
Specifically, Kachkovskiy's and Zhang--Zhao's examples are only known to satisfy $\limsup t^{-1} \|X e^{-it\J} \varphi \|  > 0$ and $\lim t^{-(1-\delta)} \|X e^{-it\J} \varphi \| = \infty$ for all $\delta > 0$, respectively~\cite{kach,zhangzhao:ballistic}.

\item One may be tempted to think that limit-periodic Jacobi matrices share many characteristics with periodic Jacobi matrices, and that one can prove such statements by ``pushing'' spectral and dynamical characteristics through to the limit using uniform convergence. However, one cannot do this in general; indeed, the spectral and quantum dynamical characteristics of limit-periodic operators are, generically, vastly distinct from those of periodic operators. For example:
\begin{itemize}
\item Limit-periodic Schr\"odinger operators ($a_n \equiv 1$) generically exhibit purely singular continuous spectrum on a Cantor set of zero Lebesgue measure and may even possess purely singular continuous spectrum with zero Hausdorff dimension and positive Lyapunov exponents~\cite{avila}.
\item There is a dense subset of the space of all limit-periodic Schr\"odinger operators upon which the spectral type is pure point~\cite{DamGorLP}; moreover, the eigenfunctions of such limit-periodic operators may exhibit Anderson localization in a remarkably strong sense~\cite{damanikgan2,poschel83}.
\item A limit-periodic operator may have an integrated density of states that enjoys no positive H\"older modulus of continuity~\cite{gankrueger}.
\item There is a dense set of limit-periodic operators for which one has $\beta^-_{\delta_0}(p) = 0$ for all $p > 0$~\cite{DLY}.
\end{itemize}

\item Theorem~\ref{pt:ballistic} is a discrete, one-dimensional counterpart to~\cite[Theorem~1.2]{KLSS}. However, the techniques of proof are distinct, and the conclusion in the one-dimensional setting is stronger; concretely, we obtain ballistic growth of all moments without time-averaging and for all initial states with sufficiently rapid decay at $\pm \infty$, whereas~\cite{KLSS} proves ballistic growth of the time-averaged second moment for a restricted class of initial states.

\item One can prove a result for limit-periodic CMV matrices that is analogous to Theorem~\ref{pt:ballistic} via very simple modifications of the proof. In this setting, one considers a (two-sided) CMV matrix $\CMV$ whose Verblunsky coefficients are bounded away from the unit circle and are exponentially well approximated by periodic sequences, and one evolves the position operator via
\[
X_{\CMV}(\ell)
=
\CMV^{-\ell} X \CMV^\ell,
\quad
\ell \in \Z.
\]
In particular, one may substitute~\cite[Section~9]{DFO} for \cite[Section~2]{DLY}, and the proof that
\[
Q_\CMV
=
\slim_{\ell \to \infty} \frac{1}{\ell} X_\CMV(\ell)
\]
exists and has the desired properties is essentially unchanged. In fact, a few estimates are simpler, since $\CMV^\ell$ is a finite-range operator for all $\ell \in \Z$. Using the Cantero--Gr\"unbaum--Moral--Vel\'azquez connection~\cite{CGMV,CGMV2}, one may make a corresponding statement for one-dimensional coined quantum walks. In particular, this permits one to extend the result of~\cite[Theorem~4]{AVWW} to 1D coined quantum walks that are uniformly approximated by translation-invariant quantum walks sufficiently rapidly; namely, such quantum walks exhibit ballistic transport in the sense of~\eqref{eq:asch-knauf}.
\end{enumerate}

\end{remark}

\subsection{Quantum Spin Systems}

Our result may be applied to a class of quantum spin systems with limit-periodic coupling coefficients to deduce lower bounds on the Lieb--Robinson velocity.

Concretely, one may define the \emph{isotropic} $XY$ \emph{spin chain} as follows. For each interval $\Lambda = [m,n] \cap \Z$, define the state space
\[
\Hi_\Lambda
\eqdef
\bigotimes_{j \in \Lambda} \C^2.
\]
Given $S \subseteq \Lambda$, we define the algebra of observables on $S$ by
\[
\mathcal A_S
\eqdef
\bigotimes_{j \in S} \mathcal A_j,
\]
where $\mathcal A_j = \C^{2 \times 2}$, acting on the $j$th coordinate of the tensor product. Given a matrix $A \in \C^{2 \times 2}$, we denote by $A_j$ the operator on $\mathcal H_\Lambda$ that acts by $A$ on the $j$th coordinate and by the identity on other coordinates. That is to say
\[
A_j
\eqdef
\I^{\otimes(j-m)} \otimes  A \otimes \I^{\otimes(n-j)}.
\]
Given sequences $\set{\mu_j}_{j \in \Z}$, $\set{\nu_j}_{j \in \Z} \in \ell^\infty(\Z,\R)$, we may define a local Hamiltonian $H_\Lambda$ on $\Hi_\Lambda$ by
\[
H_\Lambda
\eqdef
-\sum_{j =m}^{n-1} \mu_j \left(\sigma_j^x \sigma_{j+1}^x + \sigma_j^y \sigma_{j+1}^y \right)
- \sum_{j=m}^n \nu_j \sigma_j^z,
\]
where $\sigma^x$, $\sigma^y$, and $\sigma^z$ denote the Pauli matrices
\[
\sigma^x
=
\begin{bmatrix} 0 & 1 \\ 1 & 0 \end{bmatrix},
\quad
\sigma^y
=
\begin{bmatrix} 0 & -i \\ i & 0 \end{bmatrix},
\quad
\sigma^x
=
\begin{bmatrix} 1 & 0 \\ 0 & -1 \end{bmatrix}.
\]
For such models, one has the celebrated \emph{Lieb--Robinson bound} on the rate at which correlations may propagate. Here, and in what follows, $[\cdot,\cdot]$ denotes the commutator $[S,T] \eqdef ST - TS$.

\begin{theorem}[Lieb--Robinson \cite{LR72}, Nachtergaele--Sims~\cite{NS2009}]
For any subinterval $\Lambda = [m,n] \cap \Z$, any disjoint $S_1, \, S_2  \subseteq \Lambda$, any $A \in \mathcal A_{S_1}$, $B \in \mathcal A_{S_2}$ and all $t \in \R$,
\begin{equation} \label{eq:lrbound}
\norm{\left[ \tau_t^\Lambda(A),B \right]}
\leq
C\norm{A} \norm{B} e^{-c(d - v|t|)},
\end{equation}
where $\tau_t^\Lambda$ denotes the associated Heisenberg evolution on $\mathcal A_\Lambda$ with respect to $H_\Lambda$:
\[
\tau_t^\Lambda(A)
\eqdef
e^{itH_\Lambda} A e^{-itH_\Lambda},
\quad
t \in \R.
\]
In~\eqref{eq:lrbound}, $c,v > 0$ are uniform constants, $C > 0$ is a constant that depends solely on $S_1$ and $S_2$, and $d$ denotes the minimal separation between $S_1$ and $S_2$:
\[
d
=
\min\set{|n_1 - n_2| : n_1 \in S_1, \, n_2 \in S_2 }.
\]

\end{theorem}

Theorem~\ref{pt:ballistic} allows us to deduce ballistic spreading for an isotropic $XY$ chain with limit-periodic couplings that are sufficiently rapidly approximated by periodic sequences. Since the Lieb--Robinson bound is (by its very nature) an upper bound, by a lower bound on the Lieb--Robinson velocity, we mean that there exists $v_0$ such that if a bound of the form \eqref{eq:lrbound} holds with some constants $C, \, c, \,v > 0$, then one must have $v \geq v_0$.

\begin{theorem} \label{t:lr}
Suppose $\mu = \set{\mu_j}$ and $\nu = \set{\nu_j}$ are given with $\inf_j \mu_j > 0$. There is a constant $\eta_0 > 0$ that depends only on $\|\mu\|_\infty$, $\|\mu^{-1}\|_\infty$, and $\| \nu \|_\infty$ such that if there exist $q_n$-periodic sequences $\mu^{(n)}, \nu^{(n)} \in \ell^\infty(\Z)$ with
\[
\lim_{n \to \infty} e^{\eta_0 q_{n+1}} \norm{\mu - \mu^{(n)} }_\infty
=
\lim_{n \to \infty} e^{\eta_0 q_{n+1}} \norm{\nu - \nu^{(n)}}_\infty
=
0
\]
then there exists $v_0 > 0$ such that if a Lieb--Robinson bound of the form~\eqref{eq:lrbound} holds for some constants $C, \, c, \, v > 0$, then $v \ge v_0$.
\end{theorem}

\begin{proof}
This follows immediately from the arguments that prove \cite[Theorem~1.3]{DLY} by using the Jordan--Wigner transform to reduce to a free Fermion system~\cite{LSM61}. More specifically, the statement of the theorem follows from Theorem~\ref{pt:ballistic} and \cite[Theorem~6.1]{kach}. There are two minor technicalities. We allow $\set{\mu_j}$ to be nonconstant, and we only have $\ker(Q) \cap \ell^1(\Z) = \set{0}$. However, it is straightforward to check that Kachkovskiy's proof works equally well in this setting.
\qed \end{proof}

\bigskip

The structure of the paper is as follows. In Section~\ref{sec:conrate}, we prove quantitative estimates on the rate at which the left-hand side of \eqref{eq:asch-knauf} converges to the right hand side. It is quite easy to see that naive estimates in the generic situation (all spectral gaps open) yield a convergence rate proportional to $t^{-1}$. However, the constant of proportionality depends on the length of the smallest gap, and the naive estimates completely break when a gap degenerates; moreover, it is difficult to divine the length of the smallest gap from coefficient data alone, so we must work harder to prove convergence estimates that ignore gap lengths. We apply this analysis in Section~\ref{sec:proof} to prove Theorem~\ref{pt:ballistic}. Appendix~\ref{sec:ests} contains a general estimate to the effect that Heisenberg evolutions of the position operator with respect to different Jacobi matrices may diverge from one another at most quadratically in time. This is certainly well-known and not new, but we could not find a reference that stated the desired bounds in the necessary form, so we opted to present proofs to keep the paper self-contained. Finally, in Appendix~\ref{sec:acspec}, we provide a short proof that elements of $\EC(\eta)$ have purely absolutely continuous spectrum when $\eta$ is sufficiently large.

\section{Convergence Rate} \label{sec:conrate}

The main result of this section controls the rate at which the left-hand side of \eqref{eq:asch-knauf} converges to the right-hand side. In order to accomplish this, we will need to briefly recall some notions from Floquet Theory; for proofs and further details, the interested reader is referred to the excellent references~\cite[Chapter~5]{simszego} and \cite[Chapter~7]{teschljacobi}. Throughout the paper, let $\T \eqdef \R/2\pi \Z$ denote the circle, viewed as the interval $[0,2\pi]$ with endpoints identified. As a compact topological group, $\T$ is equipped with a unique translation-invariant probability measure, which we denote by $\Leb$, i.e.,
\[
\Leb(S)
=
\int_\T \chi_S(\theta) \, \frac{\dd\theta}{2\pi}
\]
for measurable subsets $S \subseteq \T$. Assuming $q \ge 3$, for each $\theta \in \T$, let
\begin{equation} \label{eq:jtheta:def}
\J_\theta
\eqdef
\begin{bmatrix}
b_1 & a_1 &&& e^{-i\theta} a_q \\
a_1 & b_2 & a_2 && \\
& \ddots & \ddots & \ddots & \\
&& a_{q-2} & b_{q-1} & a_{q-1} \\
e^{i \theta} a_q &&& a_{q-1} & b_q
\end{bmatrix}.
\end{equation}
The matrix $\J_\theta$ degenerates for $q\le 2$; for example, when $q = 2$, one must define
\[
\J_\theta
=
\begin{bmatrix}
b_1 & a_1 + e^{-i\theta}a_2 \\
a_1 + e^{i\theta} a_2 & b_2
\end{bmatrix}
\]
for the discussion that follows to hold true. Now, denote the eigenvalues of $\J_\theta$ by
\[
\lambda_1(\theta)
\leq
\lambda_2(\theta)
\leq
\cdots
\leq
\lambda_q(\theta),
\]
enumerated according to their multiplicities. Whenever $\theta \notin \set{0,\pi}$, $\J_\theta$ has simple spectrum~\cite[Theorem~5.3.4.(ii)]{simszego}. We may conjugate $\J$ to the direct integral of $\set{\J_\theta}_{\theta \in \T}$ via a $q$-adic variant of the Fourier transform. Concretely, denote by $\mathcal H_q$ the direct integral of $\C^q$ over $\T$, i.e.,
\begin{equation} \label{eq:directint}
\begin{split}
\Hi_q
=
\int^\oplus_{\T} \C^q \, \frac{\dd\theta}{2\pi}
& =
L^2\left(\T,\C^q  \right) \\
& \eqdef
\set{\vec f:\T \to \C^q : \int_\T \norm{\vec f(\theta)}^2_{\C^q} \, \frac{\dd\theta}{2\pi}
<
\infty}.
\end{split}
\end{equation}
One may define a linear operator $\mathcal F_q : \ell^2(\Z) \to \mathcal H_q$ via
\[
\mathcal F_q:\delta_{j + \ell q}
\mapsto
e^{-i\ell \theta} \vec e_j,
\quad
\ell \in \Z, \; 1 \leq j \leq q,
\]
where $\{\vec e_j : 1 \leq j \leq q\}$ denotes the standard basis of $\C^q$. It is straightforward to check that $\mathcal F_q$ defines a unitary operator, and that $\mathcal F_q \J \mathcal F_q^* = M_{\J}$, where $M_\J$ denotes the multiplication operator
\[
(M_\J \vec f \,)(\theta)
=
\J_\theta \vec f(\theta),
\quad
\vec f \in \mathcal H_q, \; \theta \in \T.
\]
Similarly, the coefficients of the operator $A \eqdef i[\J,X]$ are $q$-periodic, and hence $A$ is also diagonalized by $\mathcal F_q$. More explicitly, it is easy to check that $\mathcal F_q A \mathcal F_q^* = M_A$, where
\[
A_\theta
\eqdef
\begin{bmatrix}
0 & ia_1 &&& -ie^{-i\theta} a_q \\
-ia_1 & 0 & ia_2 && \\
& \ddots & \ddots & \ddots & \\
&& -ia_{q-2} & 0 & ia_{q-1} \\
ie^{i \theta} a_q &&& -ia_{q-1} & 0
\end{bmatrix}.
\]
In view of the identification in~\eqref{eq:directint}, we identify $M_\J$ and $M_A$ with direct integrals and write
\[
M_\J
=
\int_\T^\oplus \J_\theta \frac{\dd \theta}{2\pi},
\quad
M_A
=
\int_\T^\oplus A_\theta \frac{\dd \theta}{2\pi}.
\]
For our purposes, the key calculation from \cite{DLY} identifies $\mathcal F_q Q \mathcal F_q^*$ as a direct integral
\[
\mathcal F_q Q \mathcal F_q^*
=
\int_\T^\oplus Q_\theta \, \frac{\dd \theta}{2\pi} 
\]
by decomposing $A_\theta$ in an eigenbasis of $\J_\theta$. More precisely, let $\theta \mapsto \vec v_j(\theta)$, $1 \leq j \leq q$ be a family of measurable maps $\T \to \C^q$ such that $\set{\vec v_j(\theta): 1 \le j \le q}$ is an orthonormal basis of $\C^q$ for all $\theta \in \T$,
\[
\J_\theta \vec v_j(\theta) 
= 
\lambda_j(\theta)\vec v_j(\theta),
\quad
\text{for all } 1 \leq j \leq q, \; \theta \in \T,
\]
and such that $\vec v_j$ is smooth on the arcs $(0,\pi)$ and $(\pi,2\pi)$. Next, let $P_j(\theta)$ denote projection onto the corresponding eigenspace:\ $\ker(\J_\theta - \lambda_j(\theta) \I)$. Then, the matrix elements of $Q_\theta$ with respect to this basis obey
\begin{equation} \label{eq:DLYcalc}
\begin{split}
\langle \vec v_j(\theta), Q_\theta \vec v_k(\theta) \rangle
& =
\left\langle \vec v_j(\theta), \left( q \sum_{\ell=1}^q \dot\lambda_\ell(\theta) P_\ell(\theta) \right)\vec  v_k(\theta) \right\rangle \\
& =
\delta_{j,k} \langle \vec v_j(\theta), A_\theta \vec  v_k(\theta) \rangle \\
& =
\lim_{t \to \infty} 
\frac{1}{t}\int_0^t e^{is(\lambda_k(\theta) - \lambda_j(\theta))} \langle\vec  v_j(\theta), A_\theta\vec  v_k(\theta) \rangle \, \dd s \\
& =
\lim_{t \to \infty}
\left\langle \vec v_j(\theta), \left( \frac{1}{t} \int_0^t e^{is\J_\theta} A_\theta e^{-is\J_\theta} \, \dd s \right) \vec v_k(\theta) \right\rangle,
\end{split}
\end{equation}
for all $\theta \notin \set{0,\pi}$, where we use a dot to denote differentiation with respect to $\theta$ \cite{DLY}. Thus, we may identify $\mathcal F_q Q \mathcal F_q^* = M_Q$, where
\begin{equation} \label{eq:qdef}
Q_\theta
=
q \sum_{j=1}^q \dot\lambda_j(\theta) P_j(\theta),
\quad
\theta \in \T \setminus \set{0,\pi}.
\end{equation}
We are now in a position to state our main convergence results:

\begin{theorem} \label{t:ak:conv}
For every $R > 0$ there exists a constant $C_1 = C_1(R) > 1$ with the following property. For any $q$-periodic Jacobi matrix $\J  \in \mathcal J(R)$, one has
\begin{equation} \label{eq:ak:A:conv}
\left(
\int_0^{2\pi} \left\| Q_\theta -  \frac{1}{t} \int_0^t e^{is\J_\theta} A_\theta e^{-is\J_\theta} \, \dd s \right\|^2 \, \frac{\dd\theta}{2\pi} \right)^{1/2}
\leq
C_1^{q} t^{-1/5}.
\end{equation}

\end{theorem}

As a consequence, this immediately gives us quantitative bounds on the rate at which $t^{-1} X_\J(t)$ converges strongly to $Q_\J$.

\begin{coro} \label{coro:ak:conv}
For every $q$-periodic Jacobi matrix $\J \in \mathcal J(R)$ and every $\varphi \in D(X)$, one has
\begin{equation} \label{eq:ak:X:conv}
\left\|  Q_\J \varphi  - \frac{1}{t} X_\J(t) \varphi \right\|
\leq
 t^{-1} \|X \varphi \|  +  C_1^q \|\varphi\|_1 t^{-1/5},
\end{equation}
where $C_1 = C_1(R)$ denotes the constant from Theorem~\ref{t:ak:conv} and $\| \varphi \|_1$ denotes the $\ell^1$ norm of $\varphi$.
\end{coro}

\begin{proof}
We know that $D(X(t)) = D(X)$ for all $t$ and
\begin{equation} \label{eq:pos:integralrep}
X(t)
=
 X + \int_0^t e^{is\J} A e^{-is\J} \, \dd s
\end{equation}
as operators on $D(X)$ \cite{DLY}. Notice that $D(X) \subseteq \ell^1(\Z)$ and that
\[
\|\mathcal F_q \varphi (\theta) \|_{\C^q}
\leq
\| \varphi \|_1
\text{ for every } \theta \in \T.
\]
Thus, comparing \eqref{eq:ak:A:conv}, \eqref{eq:ak:X:conv}, and \eqref{eq:pos:integralrep}, we see that the statement of the corollary follows immediately from Theorem~\ref{t:ak:conv}.
\qed \end{proof}

In view of \eqref{eq:DLYcalc}, if $\theta \in \T$ is such that $\lambda_j(\theta) - \lambda_k(\theta)$ is not too small, then we can perform the $\dd s$ integral in the third line and obtain good control over the rate of convergence. Thus, the main obstacle in the proof of Theorem~\ref{t:ak:conv} is to bound the size of the set of ``bad'' $\theta$'s at which $\lambda_j$ and $\lambda_k$ may be close. However, since $\theta$ can be viewed as a rotation number, one is essentially asking for estimates on the density of states measure of small neighborhoods of band edges of $\J$. 

\begin{defi}
Since it is important in what follows, let us briefly recall the definition of the density of states measure $k = k_\J$ of a periodic Jacobi matrix $\J$. For each $n \geq 2$, let us consider the restriction of $\J$ to $\ell^2([1,n]\cap \Z)$ with Dirichlet boundary conditions, that is,
\[
\J^{(n)}_{\mathrm D}
=
\begin{bmatrix}
b_1 & a_1 &&&  \\
a_1 & b_2 & a_2 && \\
& \ddots & \ddots & \ddots & \\
&& a_{n-2} & b_{n-1} & a_{n-1} \\
&&& a_{n-1} & b_n
\end{bmatrix}
\in \R^{n \times n}.
\]
It is easy to see that $\J^{(n)}_{\mathrm D}$ has simple spectrum. Concretely, if $\J_{\mathrm D}^{(n)}$ has a degenerate eigenvalue, then it admits a nontrivial eigenvector $v$ with $v_1 = 0$; however, $v_1=0$ together with the Dirichlet boundary condition forces $v$ to vanish identically, contrary to the assumption that $v$ was nontrivial. Define a point measure by
\[
k_n
=
\frac{1}{n} \sum_{j=1}^n \delta_{E_{j,n}},
\]
where $\set{E_{j,n}}_{j=1}^n$ denotes the set of eigenvalues of $\J^{(n)}_{\mathrm D}$. Then, the sequence $\{k_n\}_{n \ge 2}$ has a weak$^*$ limit as $n\to \infty$; this limit is known as the \emph{density of states measure} of $\J$, and is denoted by $k$. For further information, see \cite[Chapter~5]{simszego}.
\end{defi}

It is known that the density of states measure of a periodic Jacobi matrix has square-root divergence at the band edges, and thus, the measure of the set of $\theta \in \T$ at which one may have $|\lambda_j(\theta) - \lambda_k(\theta)|\le \delta$ scales like $\delta^{1/2}$ as $\delta \downarrow 0$. However, this estimate is far too crude for our current purposes, since we need to control the scaling behavior of the D.O.S.\ across a family of Jacobi matrices with growing period, and hence, we need some measure of control on the implicit scaling constant; the following lemma supplies the necessary refinement.

\begin{lemma} \label{l:dos:bound}
Let $R > 0$ be given. There exists a constant $C_2 = C_2(R) > 1$ such that if $\J \in \mathcal J(R)$ is $q$-periodic, then
\[
k(I)
\leq
C_2^q |I|^{1/2}
\]
for all intervals $I \subseteq \R$.
\end{lemma}

To prove this lemma, we will need to estimate band lengths of $\sigma(\J)$, which one may do by estimating $\dot \lambda_j(\theta)$ for $1 \leq j \leq q$. Since such estimates will also be needed in the proof of Theorem~\ref{pt:ballistic}, we state the bounds explicitly.

\begin{lemma} \label{l:eigderivest}
For every $R > 0$ there is a constant $C_3 = C_3(R) > 1$ with the following property. For any $q$-periodic $\J \in \mathcal J(R)$, one has
\[
\left| \dot\lambda_j(\theta) \right|
\geq
C_3^{-q} |\sin(\theta)|
\]
for every $1 \leq j \leq q$ and every $\theta \in \T$.
\end{lemma}

\begin{proof}
This is essentially \cite[Equation~(3.10)]{la92}. Last works with discrete Schr\"odinger operators, for which the off-diagonal elements obey $a_n \equiv 1$. However, it is a straightforward exercise to adapt his proof to the Jacobi setting and deduce the statement of the lemma.
\qed \end{proof}

\begin{proof}[Proof of Lemma~\ref{l:dos:bound}]
The main idea is to use explicit expressions for the (derivative of) the integrated density of states together with exponential lower bounds on band lengths to control the scaling constant. More specifically, $k$ is absolutely continuous with respect to Lebesgue measure, and \cite[Corollary~5.4.20]{simszego} identifies its Radon-Nikodym derivative as
\begin{equation} \label{eq:DOS:explicit}
\frac{\dd k}{\dd E}(E)
=
\frac{1}{\pi}
\prod_{j=1}^{q-1} |E-x_j| \left[ \prod_{j=1}^q \frac{1}{|E - \alpha_j| |E - \beta_j|} \right]^{1/2},
\quad
E \in \sigma(\J).
\end{equation}
In~\eqref{eq:DOS:explicit}, $\alpha_j$ and $\beta_j$ are defined by
\[
\sigma(\J)
=
\bigcup_{j=1}^q [\alpha_j,\beta_j],
\]
and $x_j$ denotes the $j$th zero of the derivative of the discriminant of $\J$ (which is known to satisfy $\beta_j \leq  x_j \leq \alpha_{j+1}$ for each $1 \leq j \leq q-1$). Notice that the expression in~\eqref{eq:DOS:explicit} differs slightly from the expression in \cite{simszego}, in that we phrase the products in terms of the period, $q$, rather than $\ell$, the number of open gaps in the spectrum. Of course, if the $j$th gap collapses, one has $\beta_j = x_j = \alpha_{j+1}$, so one obtains the expression from \cite{simszego} by canceling any term of the form
\[
\frac{|E - x_j|}{\sqrt{|E - \alpha_{j+1} ||E - \beta_j|}}
\]
whenever $j$ corresponds to a closed gap. Using~\eqref{eq:DOS:explicit}, it is easy to see that whenever $E \in B_r^\circ \eqdef (\alpha_r,\beta_r)$ for some $1 \leq r \leq q$, we have
\begin{align*}
\frac{\dd k}{\dd E}(E)
& =
\frac{1}{\pi}
\prod_{j=1}^{q-1} |E - x_j| 
\left[ \prod_{j=1}^q \frac{1}{|E - \alpha_j| |E - \beta_j|} \right]^{1/2} \\
&
\leq
\frac{1}{\pi} \cdot  \frac{1}{\sqrt{|E - \alpha_r| |E - \beta_r|}}.
\end{align*} 
The statement of the lemma follows immediately, since
\begin{equation} \label{eq:jacobibandest}
\beta_r - \alpha_r
\geq
2 C_3^{-q}
\end{equation}
for every $1 \leq r \leq q$ by Lemma~\ref{l:eigderivest}.  
\qed \end{proof}

\begin{proof}[Proof of Theorem~\ref{t:ak:conv}]
Let $\J \in \mathcal J(R)$ be $q$-periodic. Since $\Leb\set{0,\pi} = 0$, it suffices to bound the matrix elements
\[
\left\langle
\vec v_j(\theta),
\left(Q_\theta - 
\frac{1}{t} \int_0^t e^{is\J_\theta} A_\theta e^{-is\J_\theta} \, \dd s \right)
\vec v_k(\theta) \right\rangle,
\quad
1 \leq j,k \leq q, \; \theta \in \T \setminus \set{0,\pi}.
\]
If $j = k$, then it is easy to check that
\begin{equation} \label{eq:ak:matelt:conv:eqindices}
\left\langle \vec v_j(\theta), \left(Q_\theta - \frac{1}{t} \int_0^t e^{is\J_\theta} A_\theta e^{-is\J_\theta} \, \dd s\right) \vec v_j(\theta) \right\rangle
=
0
\end{equation}
for all $t > 0$. For $j \neq k$ and all $\theta \in \T \setminus \set{0,\pi}$, \eqref{eq:DLYcalc} yields
\begin{equation} \label{eq:ak:matelt:conv:diffindices}
\left| \left\langle \vec  v_j(\theta), \left( Q_\theta - \frac{1}{t} \int_0^t e^{is\J_\theta} A_\theta e^{-is\J_\theta} \, \dd s \right) \vec v_k(\theta) \right\rangle \right|
\leq
\frac{2R}{t|\lambda_j(\theta) - \lambda_k(\theta)|}.
\end{equation}
For each $t$, let us partition $\T$ into bad and good sets, according to whether $\lambda_j(\theta)$ can be close to $\lambda_k(\theta)$ for some $j \neq k$. More specifically, we put
\[
B_t
\eqdef
\set{\theta \in \T : \exists \, 1 \leq j < k\le q \text{ such that } |\lambda_j(\theta) - \lambda_k(\theta)| \leq \varepsilon_t },
\quad
\varepsilon_t
\eqdef
R q t^{-4/5}.
\]
For $\theta \in G_t \eqdef \T \setminus B_t$,~\eqref{eq:ak:matelt:conv:eqindices} and \eqref{eq:ak:matelt:conv:diffindices} imply
\begin{equation} \label{eq:gtconv}
\begin{split}
\norm{ Q_\theta - \frac{1}{t} \int_0^t e^{is\J_\theta} A_\theta e^{-is\J_\theta} \, \dd s }^2
& \leq
\norm{ Q_\theta - \frac{1}{t} \int_0^t e^{is\J_\theta} A_\theta e^{-is\J_\theta} \, \dd s }_{\mathrm{HS}}^2 \\
& \leq
\frac{4R^2 q^2}{t^2 \varepsilon_t^2} \\
& =
4t^{-2/5},
\end{split}
\end{equation}
where $\| M\|_{\mathrm{HS}} = \big(\tr(M^* M)\big)^{1/2}$ denotes the Hilbert--Schmidt norm of the matrix $M$. The remaining contribution is bounded above by
\begin{equation} \label{eq:btconv}
\int_{B_t} \left\| Q_\theta -  \frac{1}{t} \int_0^t e^{is\J_\theta} A_\theta e^{-is\J_\theta} \, \dd s \right\|^2 \, \frac{\dd \theta}{2\pi}
\leq
16 R^2 \cdot \Leb(B_t),
\end{equation}
so it remains to bound the Lebesgue measure of $B_t$. Since $\theta$ may be viewed as a rotation number, this is tantamount to bounding the density of states measure of the $\varepsilon_t$-neighborhood of the edges of the bands of $\sigma(\J)$. More explicitly, \cite[Theorem~5.4.5]{simszego} shows that
\[
\frac{1}{\pi q} \left|\frac{\dd \theta}{\dd E} \right|
=
\frac{\dd k}{\dd E}
\]
on $\sigma(\J)$. Combining this with Lemma~\ref{l:dos:bound}, we see that
\begin{equation} \label{eq:btmeas}
\Leb(B_t)
\leq
D^q t^{-2/5},
\end{equation}
where $D$ is a constant that depends only on $R$. The statement of the theorem follows immediately from~\eqref{eq:gtconv}, \eqref{eq:btconv}, and \eqref{eq:btmeas}.
\qed \end{proof}

\section{Proof of Ballistic Transport} \label{sec:proof}

We now have all the ingredients necessary to prove Theorem~\ref{pt:ballistic}. The overall strategy is as follows:\ given $\J \in \EC(\eta)$ with $\eta$ sufficiently large, let $\set{\J_n}_{n=1}^\infty$ be a sequence of periodic operators that converges exponentially quickly to $\J$ in operator norm. For each $n$, there is a bounded self-adjoint operator
\[
Q_n
\eqdef
\slim_{t \to \infty} \frac{1}{t} X_{\J_n}(t).
\]
We begin by showing that $\{Q_n\}_{n=1}^\infty$ is a strongly Cauchy sequence of bounded operators, and hence has a limit $Q = \slim_n Q_n$, which serves as a natural candidate for $\slim_{t\to \infty} t^{-1} X_\J(t)$. The estimates from the previous section and Appendix~\ref{sec:ests} ensure that we may choose a sequence of times scales so that on the $n$th time scale, $X_\J$ is close to $X_{\J_n}$ and $X_{\J_n}$ is close to $Q_n$, thus enabling the interchange of $n$ and $t$ limits.

Dealing with the kernel of $Q$ is slightly more delicate. For $\varphi \in \ell^1(\Z) \setminus \set{0}$, we show that $Q_n \varphi$ converges to $Q \varphi$ faster than $\|Q_n \varphi\|$ can decay to zero, which forces $Q \varphi \neq 0$. The precise details follow.

\begin{proof}[Proof of Theorem~\ref{pt:ballistic}]
 Let $R$ be given, and let $\J \in \mathcal J(R) \cap \EC(\eta_0)$, where $\eta_0$ is sufficiently large. As the argument progresses, it will be easy to see that $\eta_0$ only depends on $R$ and that it will only increase finitely many times over the course of the argument.  A bit more concretely, fix a constant $\kappa > 10$, and enlarge $R$ enough so that $\J_k \in \mathcal J(R)$ for every $k$. It then suffices to take
\[
\eta_0
>
2\kappa \log(C_0),
\] 
where $C_0 = \max(6,R+1,C_1,C_3)$, $C_1=C_1(R)$ from Theorem~\ref{t:ak:conv}, and $C_3 = C_3(R)$ from Lemma~\ref{l:eigderivest}.
\newline

\noindent \textbf{Existence of {\boldmath$Q$}.} First, let us show that there exists a bounded operator $Q$ such that $Q_n \to Q$ in the strong operator topology. Since the family $\set{Q_n}_{n=1}^\infty$ is uniformly bounded (indeed, $\|Q_n\| \leq 2R$ for all $n$), it suffices to prove that 
\begin{equation}\label{eq:qn:summable}
\sum_{n = 1}^\infty \|Q_n \varphi - Q_{n+1} \varphi \|
<
\infty
\end{equation}
for every $\varphi \in \ell^1(\Z)$. To that end, define $t_n = C_0^{\kappa q_{n+1}}$ for each $n \in \Z_+$, and let $\varphi \in \ell^1(\Z)$ be given. Theorem~\ref{t:ak:conv} yields 
\begin{align*}
\left\| \left( Q_n - \frac{1}{t_n} \int_0^{t_n} e^{is\J_n} A_n e^{-is\J_n} \, \dd s \right) \varphi  \right\|
& \leq
C_0^{q_n} \|\varphi\|_1 t_n^{-1/5} \\
\left\| \left( Q_{n+1}  -  \frac{1}{t_n} \int_0^{t_n} e^{is\J_{n+1}} A_{n+1} e^{-is\J_{n+1}} \, \dd s \right)   \varphi  \right\|
& \leq
C_0^{q_{n+1}} \|\varphi\|_1  t_n^{-1/5},
\end{align*}
both of which are clearly summable by our choice of $t_n$ and $\kappa$. Next, by using Theorem~\ref{t:position:divergence} and $C_0 \geq R+1$, we obtain 
\[
\frac{1}{t_n}\norm{ \int_0^{t_n} \left( 
e^{is\J_n} A_n e^{-is\J_n} 
- e^{is\J_{n+1}} A_{n+1} e^{-is\J_{n+1}} \right) \, \dd s}
\leq
2C_0 t_n \|\J_n - \J_{n+1}\|.
\]
which is summable by our choice of $t_n$ and $\eta_0$. Combining these three summability statements, we obtain \eqref{eq:qn:summable} for all $\varphi \in \ell^1(\Z)$. Since $\|Q_n\| \leq 2R$ and $Q_n^* = Q_n$ for all $n \in \Z_+$, it follows that $Q_n$ converges strongly to a bounded self-adjoint operator $Q$ on $\ell^2(\Z)$.

\bigskip

\noindent \textbf{Convergence of {\boldmath $t^{-1}X_\J(t)$}.} Fix $\varphi \in D(X)$ and let $t_n = C_0^{\kappa q_{n+1}}$ as before. For all $t > 0$ and all $n \in \Z_+$, we have the estimate
\[
\left\| \frac{1}{t} X_\J(t) \varphi - Q\varphi \right\|
\leq 
\frac{1}{t} \left\|  X_\J(t) \varphi - X_{\J_n}(t)\varphi  \right\| 
+\left\| \frac{1}{t}  X_{\J_n}(t)\varphi - Q_n \varphi \right\| 
+\left\| Q_n \varphi - Q\varphi \right\|.
\]
Let us estimate each of these terms individually for $t_{n-1} \leq t < t_n$. The first term on the right-hand side may be controlled by Theorem~\ref{t:position:divergence}:
\begin{equation} \label{eq:heis:est1}
\frac{1}{t} \| X_\J(t) \varphi - X_{\J_n}(t) \varphi \|
\leq
2C_0^{\kappa q_{n+1} + 1} \| \J - \J_n \| \|\varphi\|,
\quad
t_{n-1} \leq t < t_n.
\end{equation}
The right-hand side of~\eqref{eq:heis:est1} goes to zero as $n \to \infty$ by our choice of $\eta_0$. Next, by Corollary~\ref{coro:ak:conv}, we get
\begin{equation} \label{eq:heis:est2}
\left\| \frac{1}{t} X_{\J_n}(t) \varphi - Q_n \varphi \right\|
\leq
\frac{1}{t_{n-1}} \|X \varphi \|_2 + C_0^{q_n} \|\varphi\|_1 t_{n-1}^{-1/5},
\quad
\quad
t_{n-1} \leq t < t_n.
\end{equation}
Of course, $n \to \infty$ as $t \to \infty$, and the right-hand sides of~\eqref{eq:heis:est1} and \eqref{eq:heis:est2} go to zero as $n \to \infty$, so, since $\|Q_n \varphi - Q \varphi\| \to 0$ as $n \to \infty$, we have
\[
Q\varphi
=
\lim_{t \to \infty} \frac{1}{t} X_\J(t) \varphi,
\]
as desired.

\bigskip

\noindent \textbf{Kernel of {\boldmath $Q$}.} We conclude by showing that $\ker(Q) \cap \ell^1(\Z) = \set{0}$. To that end, let $\varphi \in \ell^1(\Z) \setminus \set{0}$ be given, and assume without loss $\|\varphi\|_2 = 1$.

First, we will prove lower bounds on $\|Q_n \varphi \|$. Working with the Fourier transform, let us write
\[
[\mathcal F_{q_n} \varphi](\theta)
=
\begin{bmatrix}
\varphi_1\\
\varphi_2\\
\vdots\\
\varphi_{q_n}
\end{bmatrix}
+
e^{i\theta} 
\begin{bmatrix}
\varphi_{-q_n+1} \\
\varphi_{-q_n+2} \\
\vdots\\
\varphi_{0}
\end{bmatrix}
+
\vec r_n(\theta),
\quad
\theta \in \T.
\]
Since $\varphi \in \ell^1(\Z)$, the  remainder $\vec r_n$ satisfies
\[
\lim_{n \to \infty} \sup_{\theta \in \T} \norm{\vec r_n(\theta)}_{\C^{q_n}} 
=
0.
\] 
Consequently, there exists $N_0 = N_0(\varphi)$ such that for every $n \geq N_0$, there is a subset $S_n=S_n(\varphi) \subseteq \T$ with the following properties:
\begin{equation} \label{eq:s:conditions}
\Leb(S_n) \geq \frac{1}{4},
\quad
S_n \subseteq \left[\frac{\pi}{4},\frac{3\pi}{4} \right] \cup \left[\frac{5\pi}{4},\frac{7\pi}{4} \right],
\end{equation}
and
\begin{equation} \label{eq:fphi:lb}
\|[\mathcal F_{q_n} \varphi ](\theta)\|^2_{\C^{q_n}}
\geq
\frac{1}{2},
\quad
\text{ for all } \theta \in S_n.
\end{equation}
Denote by $\J_{n,\theta}$ the matrix of~\eqref{eq:jtheta:def} with $\J$ replaced by $\J_n$, let $\set{\lambda_{n,j}(\theta) : 1 \leq j \leq q_n}$ denote its eigenvalues, and define $Q_{n,\theta}$ by \eqref{eq:qdef} with $\lambda_j$ replaced by $\lambda_{n,j}$. Using \eqref{eq:s:conditions} and \eqref{eq:fphi:lb}, we have
\begin{equation} \label{eq:qphi:lbs}
\begin{split}
\|Q_n \varphi \|^2
& =
\int_\T \|Q_{n,\theta} \cdot [\mathcal F_{q_n} \varphi](\theta) \|^2 \, \frac{\dd \theta}{2\pi} \\
& \ge 
\int_{S_n} \|Q_{n,\theta} \cdot [\mathcal F_{q_n} \varphi](\theta) \|^2 \, \frac{\dd \theta}{2\pi} \\
& \ge 
\frac{1}{8} \left[ \inf_{1 \leq j \leq q_n, \, \theta \in S_n} \left|  \dot\lambda_{n,j}(\theta) \right|^2 \right] \\
& \geq
\frac{1}{16} C_3^{-2q_n},
\end{split}
\end{equation}
where we have used Lemma~\ref{l:eigderivest} in the final line. In light of \eqref{eq:qphi:lbs} and the initial estimates from the proof, it is impossible to have $Q \varphi = 0$. Concretely, since $\kappa > 10$, our initial estimates yield the following: there exists $N_1 = N_1(\varphi)$ such that
\[
\|Q_n \varphi - Q_{n+1} \varphi \|
\leq
C_0^{-q_{n+1}}
\]
whenever $n \ge N_1$. Thus, (very) crudely estimating the tail of the series, we have
\begin{equation} \label{eq:qphi:ubs}
\| Q_n \varphi - Q \varphi\|
\leq
\sum_{\ell=n}^\infty \|Q_\ell \varphi - Q_{\ell+1} \varphi \|
\leq
\sum_{\ell=n}^\infty C_0^{- q_{\ell+1}}
\leq
\sum_{j=q_{n+1}}^\infty C_0^{-j}
=
\frac{C_0^{-q_{n+1}}}{1-C_0^{-1}}
\end{equation}
for all $n \geq N_1$. Since $C_0 \geq \max(6,C_3)$ and $q_{n+1} \geq q_n + 1$ for all $n$, we may combine \eqref{eq:qphi:lbs} and \eqref{eq:qphi:ubs} to get
\[
\|Q \varphi\|
\geq
\|Q_N \varphi \|-\|(Q - Q_N) \varphi \|
\geq
\frac{1}{4} C_3^{-q_N} - \frac{C_0^{- q_{N+1}}}{1-C_0^{-1}}
\geq
\frac{1}{20} C_3^{-q_N}
>
0
\]
with $N = \max(N_0,N_1)$. Consequently, $Q \varphi \neq 0$, as desired.
\qed \end{proof}

\begin{appendix}

\section{General Propagation Estimates}\label{sec:ests}

In the appendix, we prove a simple upper bound on the rate at which the Heisenberg evolution of the position operator diverges  with respect to different local Hamiltonians given by Jacobi matrices. This material is classical and well-known; it is presented here for the convenience of the reader and to make the paper more self-contained. 

\begin{theorem} \label{t:position:divergence}
If $\J,\J' \in \mathcal J(R)$ for some $R>0$, then
\begin{equation} \label{eq:pos:div}
\norm{X_\J(t) - X_{\J'}(t)}
\leq
2\big( R t^2 + |t| \big) \|\J - \J'\|
\quad
\text{for all } t \in \R.
\end{equation}
\end{theorem}

The following lemma will be helpful.

\begin{lemma} \label{l:matexpdiv}
Let $A$ and $B$ denote bounded self-adjoint operators on a separable Hilbert space $\Hi$. For all $t \in \R$, one has the estimate
\[
\norm{e^{itA} - e^{itB}}
\leq 
 |t| \|A-B\|.
\]
\end{lemma}

\begin{proof}
Let us define $G(t) = I - e^{-itA} e^{itB}$ for $t \in \R$. One may readily verify that $G$ is a differentiable function from $\R$ into the space of bounded linear operators on $\Hi$, and that
\[
-i\frac{\dd G}{\dd t}(t)
=
e^{-itA}(A-B)e^{itB}.
\]
Consequently, $\|\dd G/\dd t\| \equiv \|A-B\|$. Since $G(0) = 0$, we know that
\[
G(t)
=
\int_0^t \frac{\dd G}{\dd s}(s)\, \dd s,
\]
and hence $\|G(t)\| \leq  |t|\|A-B\|$ for all $t$. Since $e^{itA}$ is unitary for every $t \in \R$, we have
\[
\|e^{itA} - e^{itB}\|
=
\| I - e^{-itA} e^{itB} \|
=
\| G(t)\|
\leq
|t|\|A-B\|,
\]
as desired.
\qed \end{proof}

\begin{proof}[Proof of Theorem~\ref{t:position:divergence}]
Let $\J,\J' \in \mathcal J(R)$ and $\varphi \in D(X)$ with $\| \varphi \| = 1$ be given; define $F_\varphi:\R \to \ell^2(\Z)$ by
\[
F_\varphi(t) = X_\J(t) \varphi - X_{\J'}(t) \varphi.
\] 
It is easy to see that $F_\varphi$ is a differentiable function with $F_\varphi(0) = 0$, and that one has
\[
-i \frac{\dd F_\varphi}{\dd t}(t)
=
e^{it\J} [\J,X] e^{-it\J} \varphi - e^{it\J'} [\J',X] e^{-it\J'} \varphi.
\]
Thus, we can estimate $\|\dd F_\varphi /\dd t\|$ from above as follows:
\begin{equation} \label{eq:Fphi:ub}
\begin{split}
\left\| \frac{\dd F_\varphi}{\dd t}(t)\right\|
\leq & \;
\| e^{it\J} [\J,X] e^{-it\J}\varphi - e^{it\J'} [\J,X] e^{-it\J}\varphi \| \\
& +
\| e^{it\J'} [\J,X] e^{-it\J}\varphi - e^{it\J'} [\J',X] e^{-it\J}\varphi \| \\
& +
\| e^{it\J'} [\J',X] e^{-it\J}\varphi - e^{it\J'} [\J',X] e^{-it\J'}\varphi \|.
\end{split}
\end{equation}
We can bound the first and third terms from above using Lemma~\ref{l:matexpdiv}:
\begin{align*}
\| e^{it\J} [\J,X] e^{-it\J}\varphi - e^{it\J'} [\J,X] e^{-it\J}\varphi \| 
& \leq 
|t| \|\J - \J'\| \|[\J,X]\|
\leq
2R|t| \|\J - \J'\|  \\
\| e^{it\J'} [\J',X] e^{-it\J}\varphi - e^{it\J'} [\J',X] e^{-it\J'}\varphi \|
& \leq 
|t| \|\J - \J'\| \|[\J',X]\|
\leq
2R|t| \|\J - \J'\| .
\end{align*} 
Next, we turn to the middle term on the right hand side of \eqref{eq:Fphi:ub}; by unitarity of $e^{it\J}$ and $e^{it\J'}$, it is equal to $\| [\J,X] - [\J',X] \|$, which may be bounded above via
\[
\| [\J,X] - [\J',X] \|
\leq
2 \|\J - \J'\|.
\]
Combining these three estimates and using the fundamental theorem of calculus, we have
\[
\| F_\varphi(t) \|
\leq
2(Rt^2 + |t|)\| \J - \J'\|
\quad
\text{for all }
t \in \R.
\] 
Since $D(X)$ is dense in $\ell^2(\Z)$, the statement of the theorem follows.
\qed \end{proof}

\section{Absolute Continuity for Jacobi Matrices in $\EC(\eta)$} \label{sec:acspec}

We will supply a fairly simple proof that elements of $\EC(\eta)$ have purely absolutely continuous spectrum for $\eta$ sufficiently large.

\begin{theorem} \label{t:EC:acspec}
For all $R > 0$, there exists $\eta_0 = \eta_0(R)$ such that if $\J \in \EC(\eta_0)$ and $\J \in \mathcal J(R)$, then the spectral type of $\J$ is purely absolutely continuous.
\end{theorem}

Even though Theorem~\ref{t:EC:acspec} is a direct spectral statement, the most pleasant proof involves tools from inverse spectral theory. In particular, the first step is to show that the spectrum of such a $\J$ is homogeneous in the sense of Carleson \cite{carleson83}, which enables us to use some powerful tools from the inverse theory. After proving that the spectrum is homogeneous, we use a soft argument in conjunction with work of Last~\cite{Last93} to prove that the Lyapunov exponent vanishes Lebesgue a.e.\ on the spectrum. Then, homogeneity and vanishing exponents then can be used in combination with the works of Remling~\cite{remling2011} and Poltoratski--Remling~\cite{poltrem09} to deduce purely a.c.\ spectrum. The details follow presently.

\begin{remark}
Let us remark in passing that vanishing Lyapunov exponents are insufficient to prove purely a.c.\ spectrum for (at least) two reasons, one trivial, and one subtle. First, the spectrum could have zero measure and hence not support any absolutely continuous measures whatsoever; of course, homogeneity implies positive Lebesgue measure, so this is not an issue in our case. The second, more subtle issue is that vanishing exponents simply tell us that the a.c.\ spectrum is essentially supported everywhere in the spectrum \cite{simon83CMP}, but we cannot use this to exclude singular spectrum; for this, we need additional information, in the form of the Poltoratski--Remling Theorem.
\end{remark}

\subsection{Homogeneity of the Spectrum}
We say that a compact set $\Sigma \subseteq \R$ is \emph{homogeneous} if there exist $\delta_0, \tau > 0$ such that
\[
|\Sigma \cap B_\delta(E)|
\ge
\delta \tau
\]
for every $E \in \Sigma$ and every $0 < \delta \le \delta_0$. Here and in everything that follows, we use $|\cdot|$ to denote the Lebesgue measure on $\R$.

\begin{theorem}\label{t:homog}
If $\J \in \EC(\eta)$ for sufficiently large $\eta$, then $\sigma(\J)$ is homogeneous.
\end{theorem}
Although this statement allows for slightly weaker approximations that those in \cite{FL15}, it still follows readily from their arguments; we will supply a proof for the convenience of the reader.

First, we recall the definition of the Hausdorff metric. Given two nonempty compact subsets $F,K \subseteq \R$, put
\begin{equation} \label{eq:hdmetric:def}
d_{\Hd}(F,K)
:=
\inf\{ \varepsilon > 0 : F \subseteq B_\varepsilon(K)  \text{ and } K \subseteq B_\varepsilon(F) \},
\end{equation}
where $B_\varepsilon(X)$ denotes the open $\varepsilon$-neighborhood of the set $X \subseteq \R$. The function $d_{\Hd}$ defines a metric on the space of (nonempty) compact subsets of $\R$, known as the \emph{Hausdorff metric}. 

We will quote two helpful preparatory results: first, that Lebesgue measure is upper-semicontinuous with respect to the Hausdorff metric on compact subsets of $\R$, and second, that the spectrum of a self-adjoint operator is a 1-Lipschitz function of the operator.

\begin{prop} \label{p:lebmsr:semicont}
Suppose that $\{F_n\}_{n=1}^\infty$ and $\{K_n\}_{n=1}^\infty$ are sequences of compact subsets of $\R$ that are convergent in the Hausdorff metric, and denote $F = \lim F_n$ and $K = \lim K_n$. Then,
$$
|F \cap K|
\geq
\limsup_{n\to\infty}|F_n \cap K_n|.
$$
\end{prop}

\begin{proof}
This is precisely \cite[Proposition~2.1]{F14}.
\qed \end{proof}

\begin{prop} \label{p:specdist}
Suppose $A$ and $B$ are bounded self-adjoint operators on the Hilbert space $\mathscr H$. Then one has
\begin{equation} \label{eq:specdist}
d_{\mathrm H}(\sigma(A),\sigma(B))
\leq
\|A-B\|.
\end{equation}
\end{prop}

\begin{proof}[Proof of Theorem~\ref{t:homog}] 
Let $\J \in \EC(\eta)$ be given. If $\J$ is periodic, the conclusion of the theorem is trivial, so assume that $\J$ is aperiodic. Let $\{\J_n\}_{n=1}^\infty$ be a sequence of periodic operators, such that $\J_n$ is $q_n$-periodic, $q_n|q_{n+1}$ and $q_n \neq q_{n+1}$ for each $n$, and
\begin{equation} \label{eq:ecdef2}
\lim_{n \to \infty} e^{\eta q_{n+1}}\| \J - \J_n\|
=
0.
\end{equation}
Notice that \eqref{eq:ecdef2} is preserved if one removes finitely many terms of the sequence $\{\J_n\}_{n=1}^\infty$ and consecutively renumbers the resulting sequence. For each $n \in \Z_+$, denote $\Sigma_n = \sigma(\J_n)$, $\Sigma = \sigma(\J)$, and choose $R$ large enough that
\begin{align*}
\J_n \in \mathcal J(R)
\text{ for every }
n \ge 1.
\end{align*}
By Lemma~\ref{l:eigderivest} there is a constant $C=C(R)$ such that every band of $\Sigma_n$ has length at least $C^{-q_n}$. Take $\eta > \log C$. Since $q_{n+1}$ grows at least exponentially quickly, we see that \eqref{eq:ecdef2} implies that
$$
\sum_{n=1}^\infty C^{q_{n+1}} \|\J_n - \J_{n+1}\|
<
\infty,
$$
so, by removing finitely many terms of the sequence $\{\J_n\}_{n=1}^\infty$, renumbering, and using $\eta > \log C$, we may assume that
\begin{equation} \label{eq:smalltail}
\sum_{n=1}^\infty
C^{q_{n+1}} \|\J_n - \J_{n+1}\| 
<
\frac{1}{10}.
\end{equation}
Put $\delta_0 = C^{-q_1}$. We will prove the following estimate:
\begin{equation} \label{eq:sbs:homog:est}
|B_\delta(x) \cap \Sigma_N|
\geq
\delta/2
\text{ for all } 
x \in \Sigma_N 
\text{ and every } 
0 < \delta \leq \delta_0
\end{equation}
for all $N \in \Z_+$. To that end, fix $N \in \Z_+$, $x \in \Sigma_N$, and $0 < \delta \leq \delta_0$. If $\delta \leq C^{-q_N}$, \eqref{eq:sbs:homog:est} is an obvious consequence of our choice of $C$, since $\delta$ is less than the length of the band of $\Sigma_N$ which contains $x$ in this case.  Otherwise, $\delta > C^{-q_N}$, and there is a unique integer $n$ with $1 \leq n \leq N-1$ such that
\begin{equation}\label{deltachoice}
C^{-q_{n+1}} 
< 
\delta 
\leq 
C^{-q_n}.
\end{equation}
The significance of $n$ arises precisely from the fact that it determines the periodic approximant that most closely corresponds to the length scale $\delta$. More precisely, by our choice of $C$, any band of $\Sigma_n$ has length at least $\delta$. Now, by Proposition~\ref{p:specdist}, there exists $x_0 \in \Sigma_n$ with
\begin{equation}
|x-x_0| 
\leq 
\|\J_n - \J_N\|
\leq
\sum_{\ell = n}^{N - 1}
\|\J_\ell - \J_{\ell+1}\|.
\end{equation}
 Using \eqref{eq:smalltail} and \eqref{deltachoice}, we deduce
\begin{align}
\nonumber
\lvert x-x_0 \rvert
& \leq 
\sum_{\ell = n}^{N - 1}
\|\J_\ell - \J_{\ell+1}\| \\
\nonumber
& < \delta C^{q_{n+1}}  \sum_{\ell= n}^{N - 1}  \|\J_\ell - \J_{\ell+1}\| \\
\nonumber
& \leq 
\delta  \sum_{\ell= n}^{N - 1} C^{q_{\ell+1}}  \|\J_\ell - \J_{\ell+1}\| \\
\label{eq:expsum2}
& < \frac{\delta}{10}.
\end{align}
Thus, there exists an interval $I_0$ with $x_0 \in I_0 \subseteq B_{\delta}(x) \cap \Sigma_n$ such that 
$$
|I_0| 
=
\delta - \frac{\delta}{10}
=
\frac{9 \delta}{10}.
$$
By subadditivity of Lebesgue measure, we have
$$
|B_\delta(x) \cap \Sigma_N|
\geq
|I_0 \cap \Sigma_n| - \sum_{\ell = n}^{N - 1}|I_0 \cap (\Sigma_\ell \setminus \Sigma_{\ell + 1})|
$$
Our choice of $C$ implies that the interval $I_0$ completely contains at most $\delta C^{q_{\ell + 1}}$ bands of $\Sigma_{\ell + 1}$ for each $\ell \geq n$. Consequently, Proposition~\ref{p:specdist} yields
\begin{align*}
|I_0 \cap (\Sigma_\ell \setminus \Sigma_{\ell+1})|
& \leq
2 (\delta C^{q_{\ell+1}} + 1)
\| \J_\ell - \J_{\ell + 1} \| \\
& \leq
4\delta C^{q_{\ell+1}} \| \J_\ell - \J_{\ell + 1} \|.
\end{align*}
Notice that the extra term in the parentheses on the first line is needed to account for possible boundary effects. Summing this over $\ell$ and estimating the result with \eqref{eq:smalltail}, we obtain
\[
\sum_{\ell = n}^{N - 1} |I_0 \cap (\Sigma_\ell \setminus \Sigma_{\ell + 1} )|
\leq
\sum_{\ell=n}^{N-1} 4 \delta C^{q_{\ell+1}} \| \J_\ell - \J_{\ell + 1} \|
< \frac{2 \delta}{5}.
\]
Putting all of this together, we have
\[
|B_\delta(x) \cap \Sigma_N|
\geq
|I_0 \cap \Sigma_n| - \sum_{\ell = n}^{N-1}|I_0 \cap(\Sigma_\ell \setminus \Sigma_{\ell + 1})|
>
\frac{9 \delta}{10} 
- \frac{2 \delta}{5}
=
\frac{\delta}{2}.
\]
This proves \eqref{eq:sbs:homog:est} for arbitrary $N \in \Z_+$. Consequently, we obtain
$$
|B_\delta(x) \cap \Sigma|
\geq
\delta/2
\text{ for all } x \in \Sigma, \text{ and } 0 < \delta \leq \delta_0,
$$
where we have used Proposition~\ref{p:lebmsr:semicont} with $F_n = B_\delta(x_n)$ and $K_n = \Sigma_n$, where $x_n \in \Sigma_n$ satisfies $x_n \to x$. Thus, $\Sigma$ is homogeneous, as promised.
\qed \end{proof}

\subsection{Vanishing Lyapunov Exponents}

Next, we will show that the Lyapunov exponent vanishes almost everywhere on the spectrum. In fact, this holds under vastly weaker assumptions on the rate of approximation. However, as mentioned before, absent additional information, the information provided by vanishing exponents is of limited utility in the determination of the spectral type, hence the need for a stronger assumption to deduce purely absolutely continuous spectrum.

Let us briefly recall the definition of the hull of a Jacobi matrix as well as the (averaged) Lyapunov exponent. Given $\J$, if $S$ denotes the shift, then $\J' = S \J S^{-1}$ is the Jacobi matrix with shifted parameters $a_n' = a_{n+1}$ and $b_n' = b_{n+1}$. The \emph{hull} of $\J$ is simply
\[
\Omega
=
\Omega_\J
\eqdef
\overline{\set{S^n \J S^{-n} : n \in \Z}},
\]
where the closure is taken in the operator norm topology. When $\J$ is limit-periodic, one may verify that $\Omega$ is a compact subset of the space of bounded linear operators on $\ell^2(\Z)$, and it is equipped with a unique shift-invariant probability measure. In fact, $\Omega$ has the structure of a totally disconnected compact monothetic topological group, and this invariant measure is precisely the normalized Haar measure on this group. For more insight and details on the hull of a limit-periodic operator, see \cite[Section~2]{avila}.\footnote{Avila works with discrete Schr\"odinger operators (for which $a_n \equiv 1$), but it only takes straightforward cosmetic adaptations to apply his discussions to general limit-periodic Jacobi matrices.}

The restriction of the shift to $\Omega_\J$ is a minimal transformation from $\Omega_\J$ to itself. Using this and a standard strong approximation argument, one can verify that $\sigma(\widetilde \J) = \sigma(\J)$ for every $\widetilde \J \in \Omega$. We denote this common spectrum by $\Sigma$.
\medskip

Given $n \ge 1$, the Jacobi matrix $\J = \J_{a,b}$, and $z \in \C$, the associated \emph{transfer matrix} is defined by
\[
A_z(n, \J)
=
\frac{1}{ a_n} \begin{bmatrix} z -  b_n & - 1 \\ a_n^2 & 0\end{bmatrix} \times
\cdots
\times
\frac{1}{a_1}\begin{bmatrix} z -  b_1 & - 1 \\ a_1^2 & 0\end{bmatrix}.
\]
Naturally, this defines an eigenfunction propagator in the sense that if $\J u = zu$, then
\[
\begin{bmatrix} u_{n+1} \\ a_n u_n \end{bmatrix}
=
A_z(n,\J) \begin{bmatrix} u_1 \\ a_0 u_0 \end{bmatrix},
\text{ for all } n \ge 1.
\]
Having defined the transfer matrices, the \emph{Lyapunov exponent} is given by
\[
L(z)
=
\lim_{n \to \infty} \frac{1}{n} \mathbb E(\log\| A_z(n,\cdot) \|), \quad
z \in \C,
\]
where $\mathbb E(\cdot)$ denotes integration against the unique invariant Borel probability measure on $\Omega$. In light of Kotani Theory~\cite{simon83CMP}, a key role is played by the set of energies at which the Lyapunov exponent vanishes:
\[
\ZL
\eqdef
\{z \in \C : L(z) = 0\}.
\]
By a straightforward argument using generalized eigenfunctions, one can check that $\ZL \subseteq \Sigma$.

\begin{theorem} \label{t:ETB:zeroLE}
Let $\J$ be limit-periodic. If $\J$ admits $q_n$-periodic approximants $\J_n$ such that $q_n \|\J - \J_n\| \to 0$ as $n \to \infty$, then $|\Sigma \setminus \ZL| = 0$, where $|\cdot|$ denotes the Lebesgue measure on $\R$.
\end{theorem}

\begin{proof}
We will apply \cite[Theorem~1]{Last93}.\footnote{Last works in the setting of discrete Schr\"odinger operators, but his theorem applies to Jacobi matrices; one can see this via cosmetic alterations to his proof.} In view of this result, it suffices to prove that
\[
\left|\Sigma \setminus \limsup_{n \to \infty} \sigma(\widetilde \J^{(n)}))\right|
=
0
\]
for every $\widetilde \J$ in the hull of $\J$, where $\widetilde \J^{(n)}$ denotes the $q_n$-periodic operator whose coefficients coincide with those of $\widetilde \J$ on $[1,q_n]$. Notice that $\J^{(n)}$ and $\J_n$ are not the same; however, one has
\[
\|\J - \J^{(n)}\|
\leq
4\|\J - \J_n\|
\]
by the triangle inequality. By shifting and taking limits,  this extends to every element of the hull. That is,
\[
\|\widetilde \J- \widetilde \J^{(n)}\|
\leq
4\|\J - \J_n\|
\]
for each $\widetilde \J \in \Omega_\J$. Consequently, since $\sigma(\widetilde \J^{(n)})$ has at most $q_n$ connected components, we may use Proposition~\ref{p:specdist} to deduce that
\begin{equation} \label{eq:perapprox:diff:meas:ub}
\left| \Sigma \setminus \sigma(\widetilde{\J}^{(n)}) \right|
\leq
8q_n \|\J - \J_n\|.
\end{equation}
But then,
\[
\Sigma \setminus\limsup_{n \to \infty} \sigma(\widetilde{\J}^{(n)})
=
\bigcup_{n=1}^\infty \bigcap_{k=n}^\infty \Sigma \setminus \sigma(\widetilde{\J}^{(n)})
\]
has Lebesgue measure zero, since the right-hand side of \eqref{eq:perapprox:diff:meas:ub} tends to zero as $n \to \infty$ by assumption. Thus, $|\Sigma \setminus \ZL| = 0$ by \cite[Theorem~1]{Last93}.
\qed \end{proof}

\subsection{Proof of the Main Result}

With all the preparatory work done, we are now in a position to prove that $\J \in \EC(\eta)$ has purely a.c.\ spectrum for $\eta$ large enough.

\begin{proof}[Proof of Theorem~\ref{t:EC:acspec}]
By the work of Last--Simon~\cite{LastSimon99} and minimality of the action of the shift on $\Omega$, every $\widetilde \J \in \Omega_\J$ has the same absolutely continuous spectrum; denote the common a.c.\ spectrum by $\Sigma_\ac$. Using Kotani Theory for Jacobi matrices (worked out by Simon in \cite{simon83CMP}), Theorems~\ref{t:homog} and \ref{t:ETB:zeroLE} imply that $\Sigma = \Sigma_\ac$. Concretely, the (Lebesgue) essential closure of a subset $S\subseteq \R$ is defined by
\[
\overline{S}^{\mathrm{ess}}
\eqdef
\{ x \in \R : |S \cap B_\delta(x)| > 0 \text{ for all } \delta > 0\}.
\]
Then, we have
\[
\Sigma_{\ac} 
=
\overline{\ZL}^{\mathrm{ess}}
=
\Sigma,
\]
where the first equality is due to \cite{simon83CMP} and the second inequality is a consequence of Theorems~\ref{t:homog} and \ref{t:ETB:zeroLE}. Then, \cite[Theorem~1.4]{remling2011}, implies that $\J$ is reflectionless on $\Sigma$. By Theorem~\ref{t:homog}, $\Sigma$ is homogeneous (hence weakly homogeneous), so all spectral measures are purely absolutely continuous by \cite[Corollary~2.3]{poltrem09}.
\qed \end{proof}

\end{appendix}

\section*{Acknowledements}
This work was supported in part by an AMS-Simons Travel Grant, 2016--2018. I am grateful to David Damanik, Mark Embree, Alex Elgart, and Milivoje Lukic for helpful discussions, and to G\"unter Stolz for helpful comments on the literature. I also thank the anonymous reviewers for helpful comments. Additionally, I am grateful to the Simons Center for Geometry and Physics for hospitality during the program ``Between Dynamics and Spectral Theory'', during which portions of this work were completed.

\end{document}